\documentclass{compositio}

\usepackage{mathpazo}
\usepackage{lmodern}

\usepackage[inline]{enumitem}

\usepackage{amssymb,amsmath,latexsym,lipsum}
\usepackage{enumitem}
\usepackage{tabulary}
\usepackage{booktabs}

\usepackage{charter}
\usepackage[title]{appendix}
\usepackage{longtable}
\usepackage{amsthm}
\usepackage{euscript}
\usepackage{color}
\usepackage{float}
\usepackage[T1]{fontenc}
\usepackage[singlespacing]{setspace}
\usepackage[colorinlistoftodos]{todonotes}
\usepackage{tikz}
\usetikzlibrary {positioning}
\definecolor {processblue}{cmyk}{0.96,0,0,0}
\usepackage{bookmark}
\usepackage{multirow,bigdelim}
\usepackage{mathrsfs}
\usepackage{natbib}
\usepackage{mathtools}
\usepackage{tikz-cd}
\usepackage{pgfplots}
\usepackage{pgf}
\usetikzlibrary{arrows}

\pgfplotsset{compat=1.16}
\renewcommand{\contentsname}{Contents\\{\footnotesize\normalfont}}

\newtheorem{theorem}{Theorem}[section]

\newtheorem{lemma}[theorem]{Lemma}
\newtheorem{corollary}[theorem]{Corollary}
\newtheorem{definition}[theorem]{Definition}
\theoremstyle{remark}
\newtheorem{example}[theorem]{Example}
\theoremstyle{remark}
\newtheorem{remark}[theorem]{Remark}

\newcommand{\Mgn}{\mathcal{M}_{g,n}}

\newcommand{\Cgn}{\mathcal{C}_{g,n}}

\newcommand{\Mgnexm}{\mathcal{M}_{g,n}^{ex, \mathbf{m}}}

\newcommand{\GMgnexm}{\Gamma\mathcal{M}_{g,n}^{ex, \mathbf{m}}}

\newcommand{\Gm}{\mathbb{G}_m}
\newcommand{\OC}{\mathcal{O}_C}
\newcommand{\Fp}{\mathbb{F}_p}
\newcommand{\N}{\mathbb{N}}

\begin{document}

\title{The Moduli Space of Twisted Exact Differential Forms on Curves in Positive Characteristic}

\author{Matthias Hippold}
\email{matthias.hippold@mail.huji.ac.il}
\address{Einstein Institute of Mathematics, The Hebrew University of Jerusalem, Edmond J.
Safra Campus, Giv’at Ram, Jerusalem, 91904, Israel }


\begin{abstract}
After fixing a pattern $\mathbf{m}$ of zeroes and poles, we introduce a Moduli stack $\GMgnexm$ over $\mathbb{F}_p$ that parametrizes smooth marked curves together with a non-zero differential form that is the differential of a meromorphic function. Furthermore, we consider the stack $\Mgnexm$ that parametrizes those divisors on smooth curves that appear as divisors of exact differential forms. By introducing a local-global principle for first order deformations of the objects that we consider, we show smoothness of these stacks and compute their dimension.

\end{abstract}
 
\maketitle
\newtheorem*{theorem*}{\normalfont\scshape Theorem}
\vspace*{6pt}\tableofcontents

\section{Introduction}
Asking which divisors on a smooth marked curve appear as divisors of a differential form is a classical question that has been studied so far in characteristic $0$ only. Moduli spaces of these divisors on curves appear as strata of Moduli spaces of abelian or meromorphic differentials and in the beginning were mainly studied from the perspective of Teichm\"uller dynamics \cite{masur1982interval}, \cite{veech1982gauss} where their dimension was computed. Later, Kontsevich and Zorich classified the components of these strata \cite{kontsevich2003connected} and Polishchuk showed their smoothness \cite{polishchuk2006moduli}. Most recently, mathematicians have started to consider compactifications of these Moduli spaces of curves with a differential form with a fixed pattern of zeroes and poles. Farkas and Pandharipande introduced a large compactification where the locus that parametrizes divisors on smooth curves that are divisors of differential forms is not a dense substack \cite{farkas2018moduli}. In \cite{bainbridge2018}, the authors described a modular interpretation of the closure of the locus of smooth curves, which means that they provided a minimal compactification. There are also logarithmic descriptions of these compactifications as in \cite{chen2025tale}. In some of these papers the differential form is part of the data of the Moduli functor, in other cases they describe only the divisor of the differential form. However, as the forgetful map exhibits the space with differential form as a $\mathbb{G}_m$-bundle over the space that just parametrizes divisors, the study of the local structure of these stacks is equivalent.

Many of the methods employed in the described papers rely heavily on complex analytic techniques and therefore do not generalize well to positive characteristic. For example, in the proof of his smoothness result, Polishchuk uses the fact that on first cohomology groups the de Rham differential induces the zero map. In characteristic $p$ however, if one restricts oneself to exact differential forms, so forms that are the differential of a meromorphic function, proving a smoothness result becomes possible. The study of such Moduli problems is further motivated by the use of a Moduli space of exact differential forms in the proof of a log smoothness result for a Hurwitz space for wildly ramified covers in \cite{hippold2026logarithmichurwitzspacesmixed} when the curve is of genus $0$. To generalize this log smoothness result further, a statement on smoothness of a Moduli space of exact differential forms on curves of general genus will be needed.

We now briefly describe our main theorem: After fixing a genus $g$ and an integer partition $\mathbf{m}=(m_1, \dots, m_n)$ of $2g-2$, such that $2g-2+n >0$ holds for stability reasons, we introduce the Moduli functor $\Mgnexm$ that parametrizes those marked curves of genus $g$, $(C, q_1, \dots, q_n)$ such that $\Omega_C(-\sum_{i=1}^n m_i q_i)$ has a non-zero exact section. By using a local-global principle stating that every choice of local first order deformations around the markings can be glued together to a global first order deformation, we prove the following theorem:
\begin{theorem*} 
    The Moduli functor $\Mgnexm$ is representable by a Deligne-Mumford stack. If this stack is nonempty, it is smooth of dimension
    \begin{align*}
        \dim(\Mgnexm)= g-2+n+\sum_{i=1}^n \left \lfloor \frac{m_i}{p} \right \rfloor
    \end{align*}
\end{theorem*}

Again, one can also consider the $\mathbb{G}_m$-torsor over $\Mgnexm$ that also parametrizes the differential form. We call this stack $\GMgnexm$ and the analogous result for this Moduli problem is the following:
\begin{theorem*} 
    The Moduli functor $\GMgnexm$ is representable by a Deligne-Mumford stack. If this stack is nonempty, it is smooth of dimension
    \begin{align*}
        \dim(\GMgnexm)= g-1+n+\sum_{i=1}^n \left \lfloor \frac{m_i}{p} \right \rfloor
    \end{align*}
\end{theorem*}

The author wants to express his gratitude to his academic advisor Michael Temkin whose suggestion to study Hurwitz spaces in positive characteristic with wild ramification motivated the author to develop the mathematical techniques of this paper. He furthermore wants to thank his Ph.D. brothers Yonatan Bachar and Simon Stojkovic for providing valuable feedback and many helpful discussions. 

This research was supported by ERC Consolidator Grant 770922 - BirNonArchGeom.
\subsection{Outline}
In Section~\ref{sec: cartier}, we recall basics about the Cartier operator that detects exact differential forms. We will use these insights in Section~\ref{sec:dimension} to show representability of the Moduli functors $\Mgnexm$ and $\GMgnexm$ and find a lower bound for the dimension of these stacks. In Section~\ref{sec:tangent_space}, we compute the tangent space of $\GMgnexm$ by introducing a local-global principle for its first order deformations and show smoothness of the stacks of our interest. Finally, in Section~\ref{sec:example}, we discuss the example of a Moduli space of non-zero exact differential forms on elliptic curves with neither zeroes nor poles. We will see that such a form exists iff the curve is supersingular. 

\section{The Cartier Operator}\label{sec: cartier}

In this section, we want to study the Cartier operator that is the main tool to study exactness of differential forms in positive characteristic. We want to start with the following observation that dates back to \cite{cartier1957nouvelle}:

We consider an algebraically closed field $k$ of characteristic $p$ and a smooth curve $C$ over $k$. Let $U$ be an open affine of $C$ and $\omega$ be a meromorphic differential form on $U$. For $t$ a local parameter at a closed point $x$, so a uniformizer of the local ring $\mathcal{O}_{C,x}$,  étale locally around $x$ the curve $C$ is isomorphic to $Spec(k[t])$ and one can write
\begin{align*}
    \omega= \sum_{i=0}^{p-1} f_i^p t^i dt
\end{align*}
for some rational functions $f_i$ in $t$. By expanding the $f_i$ into power series, we see that étale locally around $x$ one can integrate the form $\omega$ iff $f_{p-1}=0$. 

The correct version to think about exactness also for relative curves is the following:

\begin{definition}
    Let $\omega$ be a meromorphic relative differential form on a smooth relative curve $C$ over $S$. We call $\omega$ \textit{exact} iff there is an open cover $\{U_i\}_{i \in I}$ of $C$ and meromorphic functions $g_i$ on $U_i$ such that on each $U_i$ it holds that $\omega=dg_i$.
\end{definition}

\begin{remark}
    Over a large class of reasonable bases, such as $S$ being integral or $S$ being the spectrum of a local Artin algebra, the sheaf of meromorphic functions of $C$ and meromorphic relative differentials is constant. This means that in these cases, one does not have to consider open covers but can simply require $\omega$ to be the differential of a meromorphic function that is defined globally. 
\end{remark}

Cartier noted in \cite{cartier1957nouvelle} that when working over an algebraically closed field, our local observation comes from a larger picture, as there is an endomorphism between the sheaves of meromorphic differentials of $C$:
\begin{align*}
    K(C)\otimes_{\OC}\Omega_{C/k} \to K(C)\otimes_{\OC}\Omega_{C/k}
\end{align*}
where $K(C)$ is the function field of the curve that étale locally is given by
\begin{align*}
    \omega= \sum_{i=0}^{p-1} f_i^p t^i dt \mapsto f_{p-1}dt
\end{align*}
and a form is exact iff it is in the kernel of this operator.

If we work with families of curves over a general base, one might not have $p$-th roots needed for our formulation, so we need to modify our construction. For a smooth relative curve $C/S$ over a base in characteristic $p$, at first we want to recall the following diagram about different Frobenii:

\begin{center}
    \begin{tikzcd}
C
\arrow[drr, bend left, "F_{C}"]
\arrow[ddr, bend right]
\arrow[dr, "F_{C/S}"] & & \\
& C^{(p/S)} \arrow[r, "(F_S)_C"] \arrow[d ]
& C \arrow[d] \\
& S \arrow[r, "F_S"]
& S
\end{tikzcd}
\end{center}

Similarly to what we have defined for the case of the base being an algebraically closed field, in the general case we get the following operator that was studied in \cite{katz1970nilpotent}:
\begin{theorem}\label{thm:car_op}
    For $C/S$ a smooth relative curve with the base being of characteristic $p$ and $D=\sum_{i=1}^n m_i q_i$ a relative divisor on $C$, we define the divisor
    \begin{align*}
        D'=\sum_{i=1}^n \left \lceil \frac{m_i}{p}\right \rceil q_i'
    \end{align*}
    on $C^{(p/S)}$ with $q_i'$ the image of $q_i$ under the relative Frobenius $F_{C/S}$.
    There is a surjective linear operator
    \begin{align*}
        \mathfrak{c}: (F_{C/S})_* \Omega_{C/S}(D) \to \Omega_{C^{(p/S)}/S}(D')
    \end{align*}
    of $\mathcal{O}_{C^{(p/S)}}$-modules on $C^{(p/S)}$ that is defined étale locally by
    \begin{align*}
        \sum_{i=-N}^\infty a_i t^i dt \mapsto \sum_{k=-\left\lfloor \frac{N-1}{p}\right\rfloor }^\infty a_{kp-1}^{(p)} (t')^{k-1} dt'
    \end{align*}
    where $a_i^{(p)}$ is the image of $a_i$ under the absolute Frobenius map of $S$ and $t'$ is the local coordinate on $C^{(p/S)}$ corresponding to $t$. We will refer to this operator as the Cartier operator. If $S$ is the spectrum of an algebraically closed field $k$, then postcomposing with the isomorphism $C^{(p/k)} \cong C$ identifies the Cartier operator with the operator from \cite{cartier1957nouvelle} discussed earlier.
\end{theorem}
\begin{proof}
    The author has formulated a similar result in \cite[Theorem 2.26]{hippold2026logarithmichurwitzspacesmixed}. For the case of $S$ being an algebraically closed field, this is known from \cite{cartier1957nouvelle}, \cite{cartier1958questions} and \cite{seshadri1958operation}. The case for a general base and regular differentials is discussed in \citep[Section 7]{katz1970nilpotent} and from a more modern perspective, these results are also explained in \cite[Section 2.4.]{achinger2021global}. Our twisted version can be obtained from these results by twisting the operator and keeping track of the order of vanishing locally.
\end{proof}
While it is true that the Cartier operator is surjective, it is not always surjective on global sections. We want to understand when this is the case, at least in the case when the base is an algebraically closed field.

\begin{theorem}\label{thm:car_sur}
    Over an algebraically closed ground field $k$ of characteristic $p$, if 
    \begin{align*}
        deg(D) > 2g-2+\#Supp(D)
    \end{align*} 
    then the Cartier operator on global sections
    \begin{align*}
        H^0(\mathfrak{c}): H^0((F_{C/k})_*\Omega_{C}(D)) \to H^0(\Omega_{C^{(p/k)}}(D'))
    \end{align*}
    is surjective.
\end{theorem}

\begin{proof}

We have seen that the Cartier operator is surjective as a morphism of sheaves, so there is a short exact sequence
\begin{center}
       \begin{tikzcd}
0 \arrow[r] & \mathcal{N} \arrow[r] & (F_{C/k})_*\Omega_{C}(D) \arrow[r, "\mathfrak{c}"] & \Omega_{C^{(p/k)}}(D') \arrow[r] & 0
    \end{tikzcd}
\end{center}
where $\mathcal{N}$ is the kernel of the Cartier operator that consists precisely of those forms that are exact. Therefore, one way to see that the Cartier operator is surjective on global sections is to see that $H^1(C^{(p/k)},\mathcal{N} )$, the first cohomology of $\mathcal{N}$, vanishes.

As $\mathcal{N}$ consists of exact forms, one can also consider it as the image of the de Rham differential. In order to have the correct vanishing orders in the sequences we want to construct, we define the following divisors: 
\begin{align*}
    D''=\sum_{i=1}^n \left \lfloor \frac{m_i}{p} \right \rfloor q_i'
\end{align*}
and 
\begin{align*}
    \tilde{D}=\sum_{i=1}^n \tilde{m}_i q_i
\end{align*}
with $\tilde{m}_i=m_i$ whenever $p$ divides $m_i$ and $\tilde{m}_i=m_i-1$ when $p$ does not divide $m_i$.

Now, we can consider the short exact sequence
\begin{center}
       \begin{tikzcd}
 0 \arrow[r] &\mathcal{O}_{C^{(p/k)}}(D'') \arrow[r] & (F_{C/k})_*\mathcal{O}_{C}(\tilde{D}) \arrow[r, "d"] & \mathcal{N} \arrow[r] & 0
    \end{tikzcd}
\end{center}
From the long exact sequence on cohomology one gets a surjection
\begin{center}
      \begin{tikzcd}
\dots \arrow[r] & H^1(C^{(p/k)},(F_{C/k})_*\mathcal{O}_{C}(\tilde{D})) \arrow[r, "d"] & H^1(C^{(p/k)},\mathcal{N} )\arrow[r] & 0
    \end{tikzcd}
\end{center}
Using that the Frobenius map is finite and Serre duality, one sees that there are isomorphisms
\begin{align*}
     H^1(C^{(p/k)},(F_{C/k})_*\mathcal{O}_{C}(\tilde{D})) \xrightarrow{\sim} H^1(C, \mathcal{O}_{C}(\tilde{D})) \xrightarrow{\sim} H^0(C, \Omega_C(-\tilde{D}))^\vee
\end{align*}
By construction it holds that
\begin{align*}
    deg(\tilde{D})\geq deg(D)-\#Supp(D) > 2g-2
\end{align*}
and we can conclude that $H^0(C, \Omega_C(-\tilde{D}))^\vee=0$. Therefore also $H^1(C^{(p/k)},\mathcal{N} )$ vanishes and the surjectivity of the Cartier operator on global sections follows.
\end{proof}

\section{Definition and Dimension Estimate of \texorpdfstring{$\Mgnexm$}{Mgnexm} and \texorpdfstring{$\GMgnexm$}{GMgnexm}} \label{sec:dimension}
\begin{definition}
    Let $(g,n)$ such that $2g-2+n>0$ and $\mathbf{m}=(m_1, \dots m_n)$ be an integer partition of $2g-2$. The stack $\Mgnexm$ over $\mathbb{F}_p$ is the substack of $\Mgn$ whose $S$-points are marked curves $(C, q_1, \dots q_n)$ over $S$ such that $\Omega_{C/S}(-\sum_{i=1}^n m_i q_i)$ has a fiberwise non-zero section and this section, which is unique up to rescaling, is exact.  We call this stack the Moduli stack of twisted exact canonical divisors of type $\mathbf{m}$.

    By $\GMgnexm$, we denote the stack over $\mathbb{F}_p$ with $S$-points $(C, q_1, \dots q_n, s)$ such that the marked curve $(C, q_1, \dots, q_n)$ is in $\Mgn$ and $s$ is an exact, fiberwise non-zero section of the module $ H^0( \Omega_{C/S}(-\sum_{i=1}^n m_i q_i))$.  We call this stack the Moduli stack of twisted exact differential forms of type $\mathbf{m}$.
\end{definition}

At first, we want to construct $\GMgnexm$ and $\Mgnexm$ explicitly. Note that after we have shown representability of these stacks, it is clear that  $\GMgnexm$ is a $\mathbb{G}_m$-torsor over $\Mgnexm$, so $\Mgnexm$ can be considered to be the projectivization of $\GMgnexm$. In particular, from a result on dimension and smoothness for $\GMgnexm$, a dimension and smoothness result for $\Mgnexm$ will follow immediately. Our general plan is to use the Cartier operator to define a vector bundle of exact differential forms on $\Mgn$ and then cut out the locus of those exact forms with correct vanishing order at the markings by closed conditions.

\begin{remark}\label{rem: divisibility}
    One sees that for $\GMgnexm$ and $\Mgnexm$ being nonempty, it is a necessary condition that all $m_i$ satisfy $p \nmid m_i+1$. Otherwise, there already would be a local obstruction to the exactness condition, as the form $x^{kp-1}dx$ cannot be integrated on $k[x]$.

\end{remark}

We will start by constructing $\GMgnexm$. For this, we need an additional auxiliary marking $q_{aux}$ where we allow differential forms to have a pole of large enough order $m_{aux}$. We do this, so the sheaves that we will construct are in fact vector bundles and the Cartier operator between them is surjective. To make this explicit, we have to begin by introducing some notation.

 From now on, we denote by $I^+ \subset \{1, \dots, n\}$ the subset of indices where $m_i$ is non-negative and by $I^-$ its complement.  Note that at a $q_i$ with $i \in I^-$ the differential forms of our interest have poles, while at a $q_i$ with $i \in I^+$ these forms have zeroes.

 Let $m_{aux} \in \N$ be an integer such that $p \nmid m_{aux}-1 $ and it holds
 \begin{align*}
     m_{aux}-\sum_{i \in I^-} m_i &> 2g-1+n
 \end{align*}
As we will need these inequalities later, we also want to state that the following inequalities hold automatically:
\begin{align*}
     m_{aux}-\sum_{i \in I^-} m_i &> 0 \\
     \left\lceil \frac{m_{aux}}{p} \right\rceil-\sum_{i \in I^-} \left \lfloor \frac{m_i}{p}\right \rfloor &> 0 
\end{align*}
 Furthermore, let $\pi: \mathcal{C}_{g,n} \to \mathcal{M}_{g,n}$ be the universal curve and $q_{aux}: \Mgn \to \Cgn$ be an additional section away from the universal sections $q_1, \dots, q_n$. By
 \begin{align*}
     F_{\mathcal{C}_{g,n}/\Mgn}: \mathcal{C}_{g,n} \to \mathcal{C}_{g,n}^{(p/\Mgn)}
 \end{align*}
we denote the relative Frobenius of the universal curve. On the Frobenius twist of the universal curve $\Cgn^{(p/\Mgn)}$ we consider the vector bundle 
    \begin{align*}
       (F_{\mathcal{C}_{g,n}/\Mgn})_* \Omega_{\mathcal{C}_{g,n}/\mathcal{M}_{g,n}}(-\sum_{i \in I^-} m_i q_i+m_{aux}q_{aux})
    \end{align*}
    By pushing it forward to the base $\Mgn$, as in the construction of the Hodge bundle on $\Mgn$, for example in \citep[Chapter XIII.2.]{arbarello2011geometry}, we obtain a sheaf on $\mathcal{M}_{g,n}$:
    \begin{align*}
        \xi := \pi_*(F_{\mathcal{C}_{g,n}/\Mgn})_* \Omega_{\mathcal{C}_{g,n}/\mathcal{M}_{g,n}}(-\sum_{i \in I^-} m_i q_i+m_{aux}q_{aux})
    \end{align*}
    Note that this is in fact a vector bundle, as we chose $m_{aux}$ such that we are twisting by a positive divisor.
    
    Because
         \begin{center}
    \begin{tikzcd}
       \mathcal{C}_{g,n} \arrow[rr, "F_{\mathcal{C}_{g,n}/\Mgn}"] \arrow[rd, "\pi"]& & \Cgn^{(p/\Mgn)} \arrow[ld, "\pi"] \\
       & \Mgn &
    \end{tikzcd}
    \end{center}
    commutes, we see that sections of $\xi$, when pulled back along a morphism $S \to \Mgn$ that corresponds to a smooth curve $C/S$ with markings $q_1, \dots, q_n$, are given by the $H^0(\mathcal{O}_S)$-module 
    \begin{align*}    
    H^0( \Omega_{C/S}(-\sum_{i \in I^-} m_i q_i+m_{aux}q_{aux}))
    \end{align*}
    By Riemann-Roch, which we can use as the degree of the line bundle is high enough as we added $q_{aux}$, we can compute
    \begin{align*}
        rk(\xi)= g-1-\sum_{i \in I^-} m_i+m_{aux}
    \end{align*}

    By an analogous construction we construct another vector bundle $\xi'$ on $\Mgn$ with sections of the pullback $S \to \Mgn$ to be 
    \begin{align*}
        H^0( \Omega_{C^{(p/S)}/S}(-\sum_{i \in I^-} \left \lfloor \frac{m_i}{p}\right \rfloor q'_i+\left \lceil \frac{m_{aux}}{p} \right \rceil q'_{aux}))
    \end{align*}
    Again, by Riemann-Roch we can compute as $m_{aux}$ is large enough that 

    \begin{align*}
        rk(\xi') =g-1 -\sum_{i \in I^-} \left \lfloor \frac{m_i}{p} \right \rfloor +\left \lceil \frac{m_{aux}}{p} \right \rceil
    \end{align*}
    where $q_i'$ is the image of the marking $q_i$ under the relative Frobenius and similarly $q'_{aux}$. 
    The pushforward of the Cartier operator now gives a morphism of vector bundles:
    \begin{align*}
       \pi_* \mathfrak{c}: \xi \to \xi'
    \end{align*}

\begin{lemma}\label{lem:car_sur}
    The map of vector bundles
    \begin{align*}
        \pi_* \mathfrak{c}: \xi \to \xi'
    \end{align*}
    is surjective.
\end{lemma}
\begin{proof}
    We have to check that $\pi_* \mathfrak{c}$ is surjective on stalks. By the proper base change theorem, it is sufficient to see that for a local ring $R$ and a smooth $n$-marked curve $C$ over $R$ the map
        \begin{align*}
        H^0(\mathfrak{c}): H^0((F_{C/R})_*\Omega_{C}(D)) \to H^0(\Omega_{C^{(p/R)}}(D'))
    \end{align*}
    is surjective.
    
    As the source and the target of this map are finitely generated, by Nakayama's lemma this is equivalent to stating that the map is surjective after passing to the residue field. To see this, one might pass as well to the algebraic closure of the residue field and we have seen the desired result for algebraically closed fields in Theorem~\ref{thm:car_sur}.
\end{proof}

    From this lemma, we know that the kernel of $\pi_* \mathfrak{c}$ in the category of coherent $\mathcal{O}_{\mathcal{M}_{g,n}}$-modules is still a vector bundle and we can compute its rank to be
    \begin{align*}
        rk(\ker( \pi_*\mathfrak{c}))= rk(\xi )-rk(\xi')= - \sum_{i \in I^-} \left(m_i- \left \lfloor \frac{m_i}{p}\right \rfloor \right) + m_{aux}- \left \lceil \frac{m_{aux}}{p} \right \rceil 
    \end{align*}
    Sections of $ker(\pi_*\mathfrak{c})$ are precisely those sections of $\xi$ that are exact. We formalize these insights in the following theorem about its total space. For the construction of total spaces of vector bundles, see for example \citep[Exercise II.5.18]{hartshorne2013algebraic}.
    \begin{lemma}
        The total space of the vector bundle $\ker(\pi_*\mathfrak{c})$ on $\Mgn$ that we will denote by $V$, is smooth of dimension
        \begin{align*}
            dim(V)=3g-3+n- \sum_{i \in I^-} \left(m_i- \left \lfloor \frac{m_i}{p}\right \rfloor \right) + m_{aux}- \left \lceil \frac{m_{aux}}{p} \right \rceil 
        \end{align*}
    \end{lemma}
\begin{proof}
    $V$ is the total space of a vector bundle on a smooth stack, so itself is smooth. We have computed the rank of $ker(\pi_*\mathfrak{c})$ above and the dimension of $\Mgn$ is known to be $3g-3+n$.
\end{proof}
Clearly, $\GMgnexm$ is a substack of $V$. We now want to analyze further which conditions cut out $\GMgnexm$ out of $V$.
\begin{theorem} \label{thm: dim_bound}
    The stack $\GMgnexm$ is a Deligne-Mumford stack and if it is nonempty, we have the following lower bound of its dimension:
    \begin{align*}
        \dim(\GMgnexm) \geq g-1+n+\sum_{i=1}^n \left \lfloor \frac{m_i}{p}\right \rfloor 
    \end{align*}
\end{theorem}
\begin{proof}
    As we do not want to have differential forms in $\GMgnexm$ that are constantly zero, at first we have to remove the zero section from $V$. This open substack will be denoted by $V_0$. Removing the zero section does not change the dimension, so it also holds that
    \begin{align*}
        dim(V_0)=3g-3+n- \sum_{i \in I^-} \left(m_i- \left \lfloor \frac{m_i}{p}\right \rfloor \right) + m_{aux}- \left \lceil \frac{m_{aux}}{p} \right \rceil .
    \end{align*}
    Note that an element $(C, q_1, \dots, q_n,s)$ of $V_0$ is in $\GMgnexm$ iff $s$ has a zero of order at least $m_i$ at $q_i$ for each $i \in I^+$, so at those indices that correspond to zeroes, and $s$ has no pole along $q_{aux}$. We consider the formal completion of the total curve along the markings $q_i$. Locally, this completion is isomorphic to $Spf(R_{ij}[[t]])$ for some ring $R_{ij}$ where $t$ is a uniformizer of the divisor $q_i$, so a function that vanishes along $q_i$ with order one. This is possible because the curve $C$ is smooth over $S$ and the divisor $q_i$ is locally principal. On these formal completions, we can now write
    \begin{align*}
        s= \sum_{k=0}^{\infty} a_k t^k dt
    \end{align*}
    
    The condition to vanish with order at least $m_i$ along $q_i$ is equivalent to $a_k=0$ for $0 \leq k < m_i$. Note that this condition is well-defined, i.e. independent of the choice of local coordinates, as the order of vanishing of a function is. Already the condition of $s$ being exact forces $a_k=0$ when $k \equiv -1 \ mod  \ p$. Note that because the vanishing order of differential forms and integrability is independent of choices, also the number of conditions is independent of the choices we made. 
    
    We conclude, together with the insight that $p \nmid m_i+1$, which we gained from Remark~$\ref{rem: divisibility}$, that the condition of having a zero of order at least $m_i$ along $q_i$ is an intersection of 
    \begin{align*}
        m_i- \left\lfloor \frac{m_i}{p} \right\rfloor
    \end{align*}
    many hyperplanes in $V_0$.

    Similarly, we can consider the formal completion along $q_{aux}$ that locally is given by the formal scheme $Spf(R_{aux,j}[[t]])$ for some ring $R_{aux,j}$ where $t$ is a uniformizer of the divisor $q_{aux}$. Now we can write on this formal neighborhood
     \begin{align*}
        s= \sum_{k=-m_{aux}}^{\infty} a_k t^k dt
    \end{align*}

The requirement of $s$ not having a pole along $q_{aux}$ is equivalent to all $a_k=0$ when $k <0$. Still, those $a_k$ with $k \equiv -1 \ mod  \ p$ are already $0$ because $s$ is exact. Therefore, this condition is the intersection of 
\begin{align*}
    m_{aux}- \left \lceil \frac{m_{aux}}{p}\right \rceil
\end{align*}
many hyperplanes in $V_0$. By noting that $\GMgnexm$ is cut out by the intersection of hyperplanes from $V_0$, we see that it is a closed substack of $V_0$.  Because $\Mgn$ is a Deligne-Mumford stack, so is $V$. As we have constructed $\GMgnexm$ as a locally closed substack of $V$, it is Deligne-Mumford itself. Counting the number of these hyperplanes that cut out $\GMgnexm$ gives us the desired lower bound of its dimension, in the case that $\GMgnexm$ is nonempty:
\begin{align*}
    \dim \GMgnexm &\geq dim(V_0) - \sum_{i \in I^+} \left( m_i- \left\lfloor \frac{m_i}{p} \right\rfloor \right) -  \left( m_{aux}- \left \lceil \frac{m_{aux}}{p}\right\rceil \right) \\
    &=3g-3+n -\sum_{i=1}^n m_i + \sum_{i =1}^n \left\lfloor \frac{m_i}{p} \right\rfloor \\
    &= g-1+n+\sum_{i =1}^n  \left\lfloor \frac{m_i}{p} \right\rfloor
\end{align*}
In this calculation, we used that the $m_i$ sum up to $2g-2$.
\end{proof}
We want to use these insights to see that also $\Mgnexm$ is representable.
\begin{corollary} \label{cor: Mgnexm_rep}
    The Moduli functor $\Mgnexm$ is representable by a Deligne-Mumford stack.
\end{corollary}
\begin{proof}
    Clearly $\Mgnexm$ is a $\mathbb{G}_m$-quotient of $\GMgnexm$ where $\mathbb{G}_m$ acts by multiplication on the exact differential form $s$. Therefore, $\Mgnexm$ is an Artin stack. Recall that we have removed the zero section, so the action of $\mathbb{G}_m$ is proper and free. This is enough to see that also $\Mgnexm$ is Deligne-Mumford. 
\end{proof}
\section{Computation of the Tangent Space of \texorpdfstring{$\GMgnexm$}{GMgnexm}} \label{sec:tangent_space}
  In this section, we want to show smoothness of the stacks $\GMgnexm$ and $\Mgnexm$ by computing the tangent space of $\GMgnexm$. For this, let $k$ be an algebraic closure of $\mathbb{F}_p$. Note that it will be enough to show that the dimension of the tangent space of $\GMgnexm$ at all $k$-points is constant and is of the correct dimension.

  We have constructed $\GMgnexm$ as a closed substack of $V_0$, so we want to understand when a tangent vector to an object $(C, q_1, \dots, q_n,s)$ in $V_0$ also is a tangent vector to $\GMgnexm$. At first, we want to understand the tangent space of $V_0$.
    \begin{lemma} \label{lem:Vdim}
         The tangent space to a $k$-point $(C, q_1, \dots q_n, s)$ in $V_0$ is given by
    \begin{align*}
        T_{(C, q_1, \dots q_n, s)} V_0= H^1(\Omega_C^\vee(-
        \sum_{i=1}^n q_i)) \oplus H^0(ker(\pi_*\mathfrak{c})_{(C, q_1, \dots q_n)})
    \end{align*}
    So in particular the dimension of the tangent space of $V_0$ at $(C, q_1, \dots, q_n, s)$ is
    \begin{align*}
        dim_k T_{(C, q_1, \dots, q_n,s)} V_0=3g-3+n- \sum_{i \in I^-} \left(m_i- \left \lfloor \frac{m_i}{p}\right \rfloor \right) + m_{aux}- \left \lceil \frac{m_{aux}}{p} \right \rceil .
    \end{align*}
    \end{lemma}
    \begin{proof}
    This is a classical result about the tangent space of total spaces of vector bundles, compare \citep[Exercise II.5.18]{hartshorne2013algebraic}: Étale locally on $\Mgn$ it holds that
    \begin{align*}
        V_0 \cong \Mgn \times H^0(\ker(\pi_*\mathfrak{c}))
    \end{align*}
    so 
    \begin{align*}
        T_{(C, q_1, \dots, q_n, s)}V_0= T_{(C,q_1, \dots, q_n)} \Mgn \oplus T_{(C, q_1, \dots, q_n)}\ker(\pi_*\mathfrak{c})
    \end{align*}
    The tangent space of $\Mgn$ at $(C, q_1, \dots, q_n)$ is known to be $H^1(\Omega_C^\vee(-\sum_{i=1}^n q_i))$, this can for example be found in \citep[Section 3.B.]{harris1998moduli}. Because the tangent space to a point in an affine space is the affine space itself, the statement follows.
    \end{proof}

To understand better which tangent vectors to $V_0$ are also tangent to $\GMgnexm$, we will introduce a local-global principle. At first, we need to understand what we want the local deformations to be.
    
 \begin{definition} \label{def: loc_def_func}
    Let $m \in \N$ such that $p \nmid m+1$. The deformation functor $D^m_{loc}$ maps a local Artin algebra $A$ over $k$ to the set of pairs $(s, \mathcal{L})$ such that $s: Spec(A) \to Spec(A[[t]])$ is a section that reduces to the zero section corresponding to $k[[t]] \to k $ given by $t \mapsto 0$ and a Cartier divisor $\mathcal{L}$ on $Spec(A[[t]])$ that reduces to the Cartier divisor $\langle t^m \rangle$ on $Spec(k[[t]])$.

    The subfunctor $D^m_{loc, ex}$ parametrizes those pairs, where $\mathcal{L}$ is generated by the formal derivative of a function $f$ on $A[[t]]$. Deformations that deform the section to the section that corresponds to $c \in A$ and the Cartier divisor to the one that is generated by $(t-c)^m$ are called good. The subfunctor of good deformations is called $D^m_{loc, good}$. As $p \nmid m+1$, this is also a subfunctor of $D^m_{loc, ex}$.

    To summarize, we have an inclusion of functors
    \begin{align*}
        D^m_{loc, good} \subseteq D^m_{loc, ex} \subseteq D^m_{loc}
    \end{align*}

    Similarly, we define the deformation functor $D^{m_{aux}, q_{aux}}_{loc}$ for a fixed section $q_{aux}$ that reduces to the zero section. It maps a local Artin $k$-algebra to the set of Cartier divisors $\mathcal{L}$ on $Spec(A[[t]])$ that have a pole of order at most $m_{aux}$ at $q_{aux}$ and that reduces to the trivial Cartier divisor on $Spec(k[[t]])$ that is generated by $ u $ where $u$ is a unit in $k[[t]]$.

    It contains the subfunctor $D^{m_{aux}, q_{aux}}_{loc, ex}$ that takes as values only those Cartier divisors that are generated by the formal derivatives of meromorphic functions on $A[[t]]$.
\end{definition}

\begin{remark}
    As we fixed $q_{aux}$, so it is not part of the global Moduli functor that we want to study, we also do not consider any deformations of $q_{aux}$ in the local deformation functor.
\end{remark}
Now we want to compute the dimension of these local deformation spaces:
\begin{lemma}\label{lem:dim_loc}
It holds that
\begin{align*}
    &&\dim_k D^m_{loc, good}(k[\epsilon])&=1 \\
    &&\dim_k D^m_{loc, ex}(k[\epsilon]) &= m- \left \lfloor \frac{m}{p}\right \rfloor +1 \\
    &&\dim_k D^{m_{aux}, q_{aux}}_{loc,ex}(k[\epsilon ]) &= m_{aux} - \left \lceil \frac{m_{aux}}{p}\right \rceil
\end{align*}
\end{lemma}
\begin{proof}
    Deformations of the markings correspond to morphisms $k[\epsilon][[t]] \to k[\epsilon]$ with $\epsilon \mapsto c\epsilon$ for $c \in k$ which gives us the dimension of $D^m_{loc, good}(k[\epsilon])$. The vector space $D^m_{loc, ex}(k[\epsilon])$ parametrizes deformations of the marking together with deformations of the Cartier divisor $\langle t^m \rangle$ which are given by Cartier divisors
    \begin{align*}
        \mathcal{L}_{\epsilon} = \langle t^m + \epsilon\sum_{i=0}^{m-1} a_i t^i \rangle
    \end{align*}
    with $a_i=0$ whenever $i+1$ is divisible by $p$. Similarly, $D^{m_{aux}, q_{aux}}_{loc,ex}(k[\epsilon ])$ parametrizes deformations of the trivial Cartier divisor $\langle 1 \rangle$ that are given by Cartier divisors
    \begin{align*}
        \mathcal{L}_{\epsilon} = \langle 1 + \epsilon\sum_{i=1}^{m_{aux}} b_i t^{-i} \rangle
    \end{align*}
    with $b_i = 0$ whenever $p$ divides $i-1$. Note that here we do not have any section to deform.
\end{proof}

\begin{definition}
    Let $(C, q_1, \dots, q_n, s)$ be a $k$-point of $\GMgnexm$. Considering a $k[\epsilon]$-point of $\GMgnexm$ that reduces to $(C, q_1, \dots, q_n, s)$, localizing and completing around all $q_i$ for $i \in I^+$ and $q_{aux}$ gives a morphism
      \begin{align*}
        \Phi_{(C, q_1, \dots, q_n, s)}: T_{(C, q_1, \dots, q_n, s)}V_0 \to \bigoplus_{i \in I^+} D^{m_i}_{loc, ex}(k[\epsilon])  \oplus D^{m_{aux}, q_{aux}}_{loc, ex}(k[\epsilon])
    \end{align*}
\end{definition}
Note that in this definition we have implied that a differential form generates a Cartier divisor. To make sense of this, we identify $\hat{\Omega}^1_{k[\epsilon][[t]]/k[\epsilon]}$ with $k[\epsilon][[t]]$ by mapping $dt$ to $1$.

This allows us now to introduce a local-global principle which is the crucial tool to show smoothness of $\GMgnexm$:
\begin{theorem} \label{thm:ex_loc_glob_princ}
   $ \Phi_{(C, q_1, \dots, q_n, s)}$ is surjective. In other words, we can glue together any deformation of the local data to a global deformation.
\end{theorem}
\begin{proof}
     As any deformation of markings can clearly be glued together, it is enough to show the following: For each $n$-marked curve $C \in \Mgn(k)$ and collection of data
     \begin{align*}
         \{(q_i, \mathcal{L}_i) \in D^{m_i}_{loc}(k[\epsilon]) \}_{i \in I^+} \cup \{ \mathcal{L}_{aux} \in D^{m_{aux}, q_{aux}}_{loc}(k[\epsilon])\}
     \end{align*}
       there exists a deformation $\mathcal{C}/k[\epsilon]$ of $C$ and an exact differential form 
       \begin{align*}
           s \in \Omega_{\mathcal{C}/k[\epsilon]}(-\sum_{i \in I^-} m_i q_i+ m_{aux} q_{aux})
       \end{align*}such that the localization and completion of $s$ at each $q_i$ with $i \in I^+$ generates $\mathcal{L}_i$ and the localization and completion at $q_{aux}$ generates $\mathcal{L}_{aux}$. To simplify notation, in the following we will sometimes denote $q_{n+1}:= q_{aux}$ and $\mathcal{L}_{n+1}:=\mathcal{L}_{aux}$.
\paragraph*{\underline{Introducing Open Affines:}}
    Let $C =  \cup_{i=1}^{n+1} U_i$ be an affine cover of $C$, such that each $U_i$ contains no other marked points besides $q_i$. We can write $U_i= Spec(A_i)$ for some finitely generated $k$-algebra. We choose local coordinates $u_i$ at $q_i$ for each $U_i$. By $s_i :=s|_{U_i}$ we denote the restriction of $s$ to the open affines. Recall that from \citep[Theorem 1.2.4.]{sernesi2006deformations} we know that first order deformations of smooth affine schemes are trivial, so the only first order deformation of $U_i$ is 
    \begin{align*}
        \mathcal{U}_i:= Spec(A_i[\epsilon]).
    \end{align*}

    \paragraph*{\underline{Lifting Local Deformations to Open Affines:}}
   At first, we have to lift each local deformation in the codomain of $\Phi$ to $\mathcal{U}_i$. We start with the local deformations at those markings $q_i$ with $i \in I^+$ and will lift these deformations to $\mathcal{U}_i$ such that the lift has no poles on $\mathcal{U}_i$. To see that this is possible, denote by $r_i$ a point that is in $C$ but not in $U_i$. Note that by Riemann-Roch for $M$ large enough it holds that
   \begin{align*}
       h^0(\OC(-2q_i+Mr_i)) < h^0(\OC(-q_i+Mr_i)),
   \end{align*}
   so there is a regular function $h_i$ on $U_i$ with a simple zero at $q_i$. After rescaling by a unit, we can assume that under localization and completion $h_i$ maps to $u_i$. Again, with the notation from Lemma~\ref{lem:dim_loc} we see that 
   \begin{align*}
       s_i+\epsilon  \sum_{k=0}^{m-1} a_k h_i^k dh_i
   \end{align*}
   is a generator of the lift. To see that this form is exact, consider an étale neighborhood of $q_i$ that is given by $\mathbb{A}^1_{k[\epsilon]}$ where this section pulls back to 
   \begin{align*}
       (t^m+\epsilon \sum_{k=0}^{m-1} a_kt^k) dt
   \end{align*}
   On this étale neighborhood, it holds that 
   \begin{align*}
       \mathfrak{c} (s_i+\epsilon  \sum_{k=0}^{m-1} a_k h_i^k dh_i)=0
   \end{align*}
   and because this image is dense, exactness follows. Similarly, 
   \begin{align*}
       h^0(\OC(Mr_i)) < h^0(\OC(q_{aux}+Mr_i)),
   \end{align*}
   so there is a function $h_{aux}$ on $U_{aux}$ with only one simple pole at $q_{aux}$ and no other poles. With this, we can construct an exact lift of the deformation at the auxiliary Cartier divisor in the same way.
   \paragraph*{\underline{Gluing Local Deformations: Definition of Cocycle:}}
   In the following, we denote these generators of the corresponding Cartier divisors on the open affines that we have just constructed  by $\overline{s}_i=(s_i+\epsilon \tilde{s_i})$ with some exact differential form $\tilde{s}_i$. For indices $i \in I^-$ that correspond to poles, we take the trivial deformation coming from $\tilde{s}_i =0$. 
 We now have to find a first order deformation $\mathcal{C}$ of the curve $C$ such that the $\overline{s}_i$ glue together. Note that because $C$ is separated also $U_{ij}:= U_i \cap U_j$ is affine so we can write $U_{ij}=Spec(B_{ij})$ for some finitely generated $k$-algebra $B_{ij}$. The curve $C$ is smooth, so $\Omega_{{k(C)}/k}$ is a one-dimensional vector space over $k(C)$ that is generated by $du$ with $u$ being a meromorphic function with non-zero derivative. By definition, there are meromorphic functions $F_i, F_j, G_i, G_j$ on $C$ such that on $U_{ij}$ we can write
 \begin{align*}
     s_i|_{U_{ij}}&= dF_i =F_i' du\\
     \tilde{s_i}|_{U_{ij}}&= dG_i =G'_i du \\
     s_j|_{U_{ij}} &= dF_j =F_j' du \\
     \tilde{s_j}|_{U_{ij}}&= dG_j =G'_j du 
 \end{align*}
for some rational functions $F_i', G'_i, F_j'$ and $G'_j$ on $C$. Note that because $s_i$ and $s_j$ glue by assumption it holds $dF_i=dF_j$. Furthermore, we can choose $G_i, G_j$ to not have poles on $U_{ij}$.

First order deformations of $C$ correspond to a collection of isomorphisms $\varphi_{ij}: B_{ij}[\epsilon] \to B_{ij}[\epsilon]$ that satisfy the cocycle condition and reduce to the identity modulo $\epsilon$. This is a classical result about smooth varieties and essential for the study of $\Mgn$. This result can for example be found in \citep[Proposition 27.7.]{hartshorne2010deformation} and is the reason why first order deformations of $C$ correspond to elements in $H^1(\Omega_C^\vee)$. Such morphisms are given by 
    \begin{align*}
        \varphi_{ij}: B_{ij}[\epsilon] &\to B_{ij}[\epsilon] \\
        b_1+\epsilon  b_2 &\mapsto b_1 +\epsilon(b_2 +\eta_{ij}(b_1)) 
    \end{align*}
    with $\eta_{ij} \in Der_k(B_{ij}, B_{ij})=Hom_{B_{ij}}(\Omega_{B_{{ij}}/k}, B_{ij})$. 

    We set $\varphi_{ij}$ to be the collection of gluing morphisms that come from a collection of derivations  $\eta_{ij}$ that is the restriction of scalars to $B_{ij}$ of the $k(C)$-linear map defined by
    \begin{align*}
        du \mapsto\frac{G_j-G_i}{F'_i}
    \end{align*}
    Note that the $G_i$ and $G_j$ do not have poles and $F'_i$ has no zeroes on $U_{ij}$. Therefore the image is a regular function there, so indeed is in $B_{ij}$. We claim that these derivations give us the desired gluing data that we need to glue the $\overline{s_i}$ together. At first, note that this collection satisfies the cocycle condition because on $U_i \cap U_j \cap U_k$ we have $dF_i =dF_j=dF_k$, so it holds that
    \begin{align*}
        \eta_{ij}(u)+\eta_{jk}(u)= \frac{G_j-G_i}{F'_i} + \frac{G_k-G_j}{F'_j}= \frac{G_k-G_i}{F'_i} = \eta_{ik}(u)
    \end{align*}
    This means that the $\eta_{ij}$ we defined indeed give us a first order deformation $\mathcal{C}$ of the curve $C$.
   
    So once we have that $\varphi_{ij}^*(\overline{s}_i)-\overline{s}_j=0$, so $\overline{s}_i$ and $\overline{s}_j$ agree on the overlap $\mathcal{U}_{ij}:=\mathcal{U}_i \cap \mathcal{U}_j$, the desired local-global principle follows. 
       \paragraph*{\underline{Gluing Local Deformations: Calculation:}}
    Note that it is enough to show this equality on the completed local ring of a point in $\mathcal{U}_{ij}$ because localization and completion are injective. Let $q$ be such a point and $u_q$ a local parameter of $q$ when $q$ is considered a point of $C$, such that on the completed local ring of $q$ as a point of $C$, it holds $du=du_q$. To see existence of such a point, note that because $U_{ij}$ is dense in $C$, $du$ is not constantly $0$ on $U_{ij}$. So for such a point $q$ where $du$ is not $0$, shifting $u$ by the constant $u(q)$ gives the desired parameter. On this completed local ring we can write
     \begin{align*}
     s_i|_{U_{ij}}&= \hat{F}'_i(u_q)d u_q \\
     \tilde{s_i}|_{U_{ij}}&= \hat{G}'_i(u_q) d u_q \\
     s_j|_{U_{ij}} &= \hat{F}'_j(u_q) d u_q \\
     \tilde{s_j}|_{U_{ij}}&= \hat{G}'_j(u_q) d u_q
 \end{align*}
 
 for some functions $\hat{F}_i(u_q), \hat{F}_j(u_q), \hat{G}_i(u_q)$ and $\hat{G}_j(u_q)$, which are the germs of the $F_i, F_j, G_i$ and $G_j$ at $q$. Furthermore, it holds
 \begin{align*}
     \eta_{ij}(u_q)=\frac{\hat{G}_j(u_q)-\hat{G}_i(u_q)}{\hat{F}'_i(u_q)}
 \end{align*}
 To show that on this completed local ring the equality holds, we do the following calculation, similar as in \citep[Theorem 2.1.]{bainbridge2016strata}:
     \begin{align*}
        &\varphi_{ij}^*(\overline{s_i})-\overline{s_j}=\\
        &=\varphi_{ij}^*((\hat{F}'_i(u_q)+\epsilon \hat{G}'_i(u_q))d{u_q})-\overline{s_j} \\
        &=(\hat{F}'_i(u_q+\epsilon\eta_{ij}(u_q))+\epsilon \hat{G}'_i(u_q+\epsilon \eta_{ij}(u_q))) d(u_q+\epsilon \eta_{ij}(u_q))-\overline{s_j} \\
        &=(\hat{F}'_i(u_q)+\epsilon \hat{F}''_i(u_q)\eta_{ij}(u_q)+\epsilon \hat{G}'_i(u_q))(du_q+\epsilon\eta'_{ij}(u_q) du_q)-\overline{s_j} \\
        &=(\hat{F}'_i(u_q)+\epsilon(\hat{F}''_i(u_q)\eta_{ij}(u_q)+\hat{F}'_i(u_q)\eta'_{ij}(u_q)+\hat{G}'_i(u_q))(du_q)-\overline{s_j} \\
        &=\epsilon(\hat{F}''_i(u_q)\frac{\hat{G}_j(u_q)-\hat{G}_i(u_q)}{\hat{F}'_i(u_q)}+\frac{\hat{F}'_i(u_q)(\hat{G}'_j(u_q)-\hat{G}'_i(u_q))-\hat{F}''_i(u_q)(\hat{G}_j(u_q)-\hat{G}_i(u_q))}{\hat{F}'_i(u_q)}+ \\
        &+\hat{G}'_i(u_q)- \hat{G}'_j(u_q))(du_q) \\
        &=0
    \end{align*}
   To see that the constructed section is exact, it is enough to note that by construction its image under the twisted Cartier operator is $0$ on stalks.
\end{proof}

This local-global principle is the key ingredient for our main statement:

\begin{theorem}\label{thm:stack_smooth}
    If the Moduli space $\GMgnexm$ is nonempty, it is smooth of dimension
    \begin{align*}
        \dim(\GMgnexm)= g-1+n+\sum_{i=1}^n \left \lfloor \frac{m_i}{p} \right \rfloor
    \end{align*}
\end{theorem}
\begin{proof}
   Let $(C, q_1, \dots, q_n,s)$ be a $k$-point of $\GMgnexm$. A tangent vector $\nu \in T_{(C, q_1, \dots, q_n,s)} V_0$ now is also a tangent vector to $(C, q_1, \dots, q_n,s)$ in $\GMgnexm$ if it does not split the multiplicity of zeroes and has no poles along the auxiliary marking. This is equivalent to $\nu$ being in the kernel of
    \begin{align*}
       pr \circ \Phi: T_xV_0 \to \bigoplus_{i \in I^+} (D^{m_i}_{loc, ex}(k[\epsilon]) / D^{m_i}_{loc, good} (k[\epsilon]) )\oplus  D^{m_{aux}, q_{aux}}_{loc}(k[\epsilon ])
    \end{align*}
   where $pr$ is the obvious projection map. In the last theorem we saw that this morphism is surjective. From Lemma~\ref{lem:dim_loc} and Lemma~\ref{lem:Vdim} we know the dimension of all appearing vector spaces. Using that the $m_i$ sum up to $2g-2$, we can compute the dimension of the tangent space of $\GMgnexm$ at $(C, q_1, \dots, q_n, s)$ to be 
    \begin{align*}
    \dim_k T_{(C, q_1, \dots, q_n, s)} \GMgnexm=g-1+n + \sum_{i=1}^n \left \lfloor \frac{m_i}{p}\right \rfloor
    \end{align*}

    From Theorem~\ref{thm: dim_bound} we know that the dimension of the tangent space is also at most the dimension of $\GMgnexm$, so in fact they are equal and it follows the desired smoothness result.
    \end{proof}
We can further conclude the analogous statement for $\Mgnexm$.

\begin{corollary} \label{cor:Mgnexm_smth}
        If the Moduli space $\Mgnexm$ is nonempty, it is smooth of dimension
    \begin{align*}
        \dim(\Mgnexm)= g-2+n+\sum_{i=1}^n \left \lfloor \frac{m_i}{p} \right \rfloor
    \end{align*}
\end{corollary}
\begin{proof}
    From our discussion in Corollary~\ref{cor: Mgnexm_rep}, we know that $\GMgnexm$ is a $\mathbb{G}_m$-torsor over $\Mgnexm$. This means each point of $\GMgnexm$ has an étale neighborhood $U$ that is isomorphic to an open subscheme of $\Mgnexm \times \mathbb{A}^1$. By faithfully flat descent of smoothness, it follows that $\Mgnexm$ is smooth.

    To see the dimension result, it is enough to see that taking a $\mathbb{G}_m$-quotient reduces the dimension by one.
\end{proof}

\section{Example: Differential Forms on Elliptic Curve Without Zeroes and Poles} \label{sec:example}
As for a smooth elliptic curve $C$ the sheaf of differential forms $\Omega_C$ is isomorphic to the structure sheaf $\mathcal{O}_C$, there is a, up to rescaling, unique differential form on $C$ without zeroes or poles. The question of when this form is exact, was the starting point of the research that led to this paper. As this consideration provides good intuition about the Cartier operator and its answer shows a beautiful connection to questions coming from the theory of Abelian varieties, we will discuss this example here explicitly.
\begin{example}
In characteristic $p$, the, up to rescaling, unique differential form without zeroes and poles on an elliptic curve is exact, iff the curve is supersingular. 
\end{example}
\begin{remark}
    From this, it follows immediately that for $g=n=1$ and $\mathbf{m}=(0)$ the Moduli space $\Mgnexm$ is a disjoint union of points and $\GMgnexm$ is a disjoint union of copies of $\Gm$. In particular, one can see explicitly that these spaces are smooth of the correct dimension as described in Theorem~\ref{thm:stack_smooth} and Corollary~\ref{cor:Mgnexm_smth}.
\end{remark}

We will start by showing this result explicitly for $p>3$ and explain later how a more conceptual understanding can be used to see that this also holds for $p=2$ and $p=3$.

For $p>3$ a smooth elliptic curve has an affine chart $U$ given by a Weierstrass equation
\begin{align*}
    y^2=x(x-1)(x-\lambda)=x^3-(\lambda+1)x^2+ \lambda x
\end{align*}

It is a classical result that on this chart $U$ the module of regular differential forms is generated by the form $\frac{dx}{y}$ that has neither zeroes nor poles. As $U$ is dense in $C$, a differential form without zeroes and poles on $C$ is exact iff $\frac{dx}{y}$ is on $U$.

Restricted to $U$, for 
\begin{align*}
    A=\frac{\Fp[x,y]}{(y^2-x(x-1)(x-\lambda))}
\end{align*}
and
\begin{align*}
    A^{(p/\Fp)}=\frac{\Fp[x',y']}{(y'^{2}-x'(x'-1)(x'-\lambda^p))}
\end{align*}
the Cartier operator is a $A^{(p/\mathbb{F}_p)}$- linear map
\begin{align*}
    (F_{A/\Fp})_* \frac{A[dx,dy]}{(2dy-(3x^2-2(\lambda+1)x+\lambda)dx)}&= (F_{A/\Fp})_* \Omega_{A/\Fp} \to \\
    \Omega_{A^{(p/\Fp)}/\Fp} &= \frac{A^{(p/\Fp)}[dx',dy']}{(2dy'-(3x'^2-2(\lambda^p+1)x'+\lambda^p)dx')}
\end{align*}
To check when the form of interest $\frac{dx}{y}$ is exact, we have to check when the Cartier operator maps it to $0$. For this, we see that
\begin{align*}
    \mathfrak{c}\left(\frac{dx}{y}\right)=\frac{1}{y'}\mathfrak{c} \left( y^{p-1} dx \right)=\frac{1}{y'}\mathfrak{c} \left((x(x-1)(x-\lambda))^{\frac{p-1}{2}} dx \right)
\end{align*}
To understand this better, we write the polynomial in question as
\begin{align*}
    \left(x(x-1)(x-\lambda)\right)^{\frac{p-1}{2}}=\sum_{i=0}^\frac{3p-3}{2} a_i x^i
\end{align*}
From this, we see that
\begin{align*}
    \mathfrak{c}\left(\frac{dx}{y}\right)=\frac{a_{p-1}}{y'} dx'
\end{align*}
so $\frac{dx}{y}$ is exact iff $a_{p-1}=0$. But according to \citep[V.4 Theorem 4.1.]{silverman2009arithmetic} this is equivalent to the curve being supersingular. 

The more conceptual way to think about this phenomena that also works for $p=2$ and $p=3$ is the following: Clearly $H^1(C, \OC)$ is a one dimensional vector space over $\Fp$. The absolute Frobenius induces an endomorphism of $H^1(C, \OC)$ that because of dimension reasons is just multiplication by a scalar. In \citep[Section 2.0]{katz1973p}, Katz described the Hasse invariant, as precisely being that scalar. As an elliptic curve is supersingular iff its Hasse invariant is $0$, we can conclude the following lemma:
\begin{lemma}
    An elliptic curve $C$ over $\Fp$ is supersingular iff the induced endomorphism of the absolute Frobenius map on $H^1(C, \OC)$ has a nontrivial cokernel.
\end{lemma}

Note that as we work over a finite, so in particular perfect, field, we do not need to differentiate between the absolute and the relative Frobenius morphism. From \citep[Proposition 9]{serre1958topologie} we know that under Serre duality, the Cartier operator corresponds to the Frobenius morphism. By using this duality theorem, we therefore can conclude that the curve is supersingular, iff the Cartier operator maps all differential forms from $H^0(\Omega_{C/\Fp})$ to $0$. As $H^0(\Omega_{C/\Fp})$ is one-dimensional, this is equivalent to the, up to rescaling unique, differential form without zeroes and poles being exact.

\bibliographystyle{alpha}
\bibliography{bib}

\end{document}